\documentclass[11pt,twoside]{article}
\usepackage[latin1]{inputenc}
\usepackage{amsmath,amsfonts,amscd,amsthm}

\newtheorem{thm}{Théorème}[section]
\newtheorem{prop}[thm]{Proposition}

\newtheorem{cor}[thm]{Corollaire}

\theoremstyle{definition}
\newtheorem{defin}[thm]{Définition}

\theoremstyle{remark}

\numberwithin{equation}{section}

\providecommand\ufootnote[1]{{\let\thefootnote\relax\footnote[0]{#1}}}

\newcommand{\cc}{\mathcal C}
\newcommand{\dc}{\mathcal D}
\newcommand{\ec}{\mathcal E}

\newcommand{\nb}{\mathbb N}

\newcommand{\cb}{\mathbb C}
\newcommand{\zb}{\mathbb Z}

\newcommand{\ci}{{\mathcal C}^\infty}

\newcommand{\ol}{\overline}

\newcommand{\pa}{\partial}
\newcommand{\opa}{\ol\partial}
\newcommand{\wt}{\widetilde}

 \DeclareMathOperator{\im}{Im}

\begin{document}

\title{Théorie $L^p$ pour l'équation de Cauchy-Riemann}

\author{Christine LAURENT-THI\'{E}BAUT}

\date{}
\maketitle

\ufootnote{\hskip-0.6cm UJF-Grenoble 1, Institut Fourier, Grenoble,
F-38041, France\newline CNRS UMR 5582, Institut Fourier,
Saint-Martin d'H\`eres, F-38402, France
}

\ufootnote{\hskip-0.6cm  {\it 2010 A.M.S. Classification}~: 32C37,
  32F10, 32W05 .\newline {\it Key words}~: équation de Cauchy-Riemann,
  dualité de Serre, estimations $L^p$.}

\bibliographystyle{amsplain}

Cet article propose une étude systématique de l'opérateur de Cauchy-Riemann relativement aux espaces $L^p$, $1<p<+\infty$, dans les variétés analytique complexes. L'intérêt des espaces $L^p$, $1<p<+\infty$, est que leurs espaces duaux sont des espaces de même type, plus précisément le dual d'un espace $L^p$, $1<p<+\infty$,  est un espace $L^{p'}$, où $p'$ est déterminé par $\frac{1}{p}+\frac{1}{p'}=1$. Cette propriété prend toute son importance dans l'étude de la dualité de Serre.

Notons que si $p=2$, alors $p'=2$ et les espaces $L^2$ sont des espaces de Hilbert. Le cas particulier $p=2$ a été très largement étudié via la théorie $L^2$ de L. Hörmander \cite{Ho} et l'étude de l'opérateur de Neumann initiée par G.B. Folland et J.J. Kohn \cite{FoKo} (voir \cite{ChSh} et les références contenues dans sa bibliographie). Un travail récent de D. Chakrabarti et Mei-Chi Shaw \cite{ChaSh} traite de la dualité de Serre dans ce cadre.

Une première partie s'intéresse à la cohomologie de Dolbeault $L^p_{loc}$ à valeurs dans un fibré vectoriel holomorphe sur une variété analytique complexe. Nous prouvons l'existence de l'isomorphisme de Dolbeault entre la cohomologie de Dolbeault $L^p_{loc}$ et la cohomologie de Dolbeault $\ci$ ainsi que la cohomologie de Dolbeault pour les courants. Nous prouvons en particulier

\begin{thm}
Soient $X$ une variété analytique complexe de dimension complexe $n$, $E$ un fibré vectoriel holomorphe sur $X$ et $p\geq 1$. Si $r$ est un entier tel que $0\leq r\leq n$,

(i) pour toute $f\in (L^p_{loc})^{r,q}(X,E)$, $q\geq 1$, telle que $\opa f=0$, et tout voisinage $U$ du support de $f$, il existe $g\in (L^p_{loc})^{r,q-1}(X,E)$ à support dans $U$ telle que $f-\opa g\in\ec^{r,q}(X,E)$;

(ii) si $f\in (L^p_{loc})^{r,q}(X,E)$, $q\geq 1$, vérifie $f=\opa S$ pour un courant $S\in\dc'^{r,q-1}(X,E)$, alors pour tout voisinage $U$ du support de $S$ il existe $g\in (W^{1,p}_{loc})^{r,q-1}(X,E)$ à support dans $U$ telle que $f=\opa g$ sur $X$.
\end{thm}

On obtient alors que si $f\in (L^p)^{r,q}(X,E)$, $0\leq r\leq n$, $q\geq 1$, est une forme différentielle à support compact dans $X$, qui est exacte au sens des courants, il existe $g\in (W^{1,p})^{r,q-1}(X,E)$ à support compact dans $X$ telle que $f=\opa g$ sur $X$. Cela répond à une question posée par G. Tomassini (voir \cite{AmMo} et \cite{Am}).

Il en résulte également un théorème d'extension du type Hartogs-Bochner-Severi pour les fonctions CR  de régularité Sobolev $W^{1-\frac{1}{p},p}$.

Dans la seconde partie nous définissons plusieurs complexes de Cauchy-Riemann pour $L^p$. Nous développons tout d'abord la théorie d'Andreotti-Grauert dans le cadre $L^p$, une première avancée dans ce domaine est due à N. Kerzman \cite{Ke}. Pour ce faire, nous nous appuyons sur des résultats de compacité résultant des estimations Sobolev jusqu'au bord dues à R. Beals, P. Greiner et N. Stanton \cite{BGS}.

Ensuite nous précisons les propriétés de dualité reliant les différents complexes $L^p$ et en appliquant des résultats abstraits nous développons la dualité de Serre dans ce cadre. Cela nous permet d'étudier la résolution du $\opa$ à support exact. Plus précisément nous considérons le problème de Cauchy faible suivant~:
Soit $X$ une variété analytique complexe de dimension complexe $n\geq 2$, $E$ un fibré vectoriel holomorphe sur $X$ et $D$ un domaine relativement compact dans $X$. Etant donnée une $(r,q)$-forme $f$ à coefficients dans $L^p(X,E)$, avec $0\leq r\leq n$ et $1\leq q\leq n$, telle que
$${\rm supp}~f\subset\ol D \quad {\rm et}\quad \opa f=0~{\rm au~sens~des~distributions~dans~}X,$$
existe-t-il une $(r,q-1)$-forme $g$  à coefficients dans $L^p(X)$ telle que
$${\rm supp}~g\subset\ol D \quad {\rm et}\quad \opa g=f~{\rm au~sens~des~distributions~dans~}X~?$$
Notons que lorsque $q=n$, la condition $\opa f=0$ dans l'énoncé du problème de Cauchy est vide et qu'une condition de moment est nécessaire pour avoir une solution. Des conditions cohomologiques puis géométriques sur $X$ et $D$ sont données pour obtenir une réponse affirmative (voir les Théorèmes \ref{n} et \ref{infn-1} et les corollaires qui suivent). La théorie développée ici donne un bien meilleur contrôle du support de la solution en fonction du support de la donnée par rapport aux résultats démontrés dans la première partie.

Finalement nous pouvons déduire de l'étude du problème de Cauchy des résultats d'extensions pour les formes différentielles CR, de résolution  de l'équation de Cauchy-Riemann dans des anneaux et des théoremes de séparation pour certains groupes de cohomologie. 

Je voudrais remercier ici Eric Amar pour les nombreux échanges que
nous avons eu lors de la préparation de cet article.

\section{Cohomologie $L^p_{loc}$ dans les variétés complexes}\label{s1}

Soient $X$ une variété analytique complexe de dimension complexe $n$ et $E$ un fibré vectoriel holomorphe sur $X$, on note $\ec^{r,q}(X,E)$, $0\leq r,q\leq n$, l'espace de Fréchet des $(r,q)$-formes différentielles à valeurs dans $E$ de classe $\ci$ sur $X$ muni de la topologie de la convergence uniforme des formes et de toutes leurs dérivées sur tout compact de $X$ et $\dc^{r,q}(X,E)$, $0\leq r,q\leq n$, le sous espace de $\ec^{r,q}(X,E)$ des formes à support compact muni de sa topologie usuelle de limite inductive de Fréchet. On désigne par $\opa$ l'opérateur de Cauchy-Riemann associé à la structure complexe de $X$ et pour tout $0\leq r\leq n$ on considère les groupes de cohomologie $H^{r,q}(X,E)$, $0\leq q\leq n$, du complexe $(\ec^{r,\bullet}(X,E),\opa)$ et les groupes de cohomologie $H_c^{r,q}(X,E)$, $0\leq q\leq n$, du complexe $(\dc^{r,\bullet}(X,E),\opa)$. Il s'agit des groupes de cohomologie de Dolbeault et des groupes de cohomologie de Dolbeault à support compact de $X$ à valeurs dans $E$, ils sont définis par
$$H^{r,q}(X,E)=\frac{\ec^{r,q}(X,E)\cap\ker\opa}{\opa\ec^{r,q-1}(X,E)}\quad {\rm et}\quad H_c^{r,q}(X,E)=\frac{\dc^{r,q}(X,E)\cap\ker\opa}{\opa\dc^{r,q-1}(X,E)}$$
et munis de la topologie quotient correspondante.

Si $\dc'^{r,q}(X,E)$, $0\leq r,q\leq n$, désigne l'espace des courants de bidegré $(r,q)$ à valeurs dans $E$ sur $X$, on peut étendre l'opérateur $\opa$ à cet espace en posant pour $T\in\dc'^{r,q}(X,E)$, $<\opa T,\varphi>=(-1)^{p+q+1}<T,\opa\varphi>$ pour toute forme $\varphi\in\dc^{n-r,n-q-1}(X,E^*)$. On obtient ainsi un nouveau complexe $(\dc'^{r,\bullet}(X,E),\opa)$ dont les groupes de  cohomologie sont notés $H_{cour}^{r,q}(X,E)$.

Pour $0\leq r\leq n$, notons $\Omega^r_E$ le faisceau des $(r,0)$-formes holomorphes sur $X$ à valeurs dans $E$. Il résulte du lemme de Dolbeault et de l'isomorphisme de de Rham-Weil (voir par exemple le Théorème 2.14 dans \cite{HeLe2}) que les groupes de cohomologie du faisceau  $\Omega^p_E$ et les groupes de cohomologie de Dolbeault  sont isomorphes, i.e.
$$H^q(X,\Omega^r_E)\sim H^{r,q}(X,E).$$

Le lemme de Dolbeault étant également valide pour les courants (voir par exemple le Théorème 2.13 dans \cite{HeLe2}), les groupes de cohomologie du faisceau  $\Omega^p_E$ et les groupes de cohomologie au sens des courants sont aussi isomorphes, i.e.
$$H^q(X,\Omega^r_E)\sim H_{cour}^{r,q}(X,E).$$

Dans cette section, nous allons généraliser ces résultats à la cohomologie $L^p_{loc}$, $p\geq 1$. Pour faciliter la lecture de ce travail nous donnons en détail la plupart des démonstrations même lorsqu'elles sont similaires à celles de \cite{HeLe2}.

Considérons pour $p\geq 1$, le sous espace $(L^p_{loc})^{r,q}(X,E)$, $0\leq r,q\leq n$, de $\dc'^{r,q}(X,E)$ des $(r,q)$-formes différentielles à valeurs dans $E$ à coefficients dans $L^p_{loc}$ muni de la topologie de la convergence $L^p$ sur tout compact de $X$. On définit l'opérateur $\opa$ de $(L^p_{loc})^{r,q}(X,E)$ dans $(L^p_{loc})^{r,q+1}(X,E)$ de la manière suivante~: si $f\in (L^p_{loc})^{r,q}(X,E)$ et $g\in (L^p_{loc})^{r,q+1}(X,E)$, on a $\opa f=g$ si $\opa f$ et $g$ coïncident au sens des courants et son domaine de définition $Dom(\opa)$ est l'ensemble des $f\in (L^p_{loc})^{r,q}(X,E)$ telles que $\opa f\in (L^p_{loc})^{r,q+1}(X,E)$. Notons que puisque $\opa\circ\opa=0$, si $f\in Dom(\opa)$ alors $\opa f\in Dom(\opa)$. On obtient ainsi un complexe d'opérateurs non bornés $((L^p_{loc})^{r,\bullet}(X,E),\opa)$.

Le lemme de Dolbeault est encore valable pour les formes à coefficients dans $L^p_{loc}$. En effet, pour tout $x\in X$, si $U\subset\subset V$ sont deux voisinages de $x$ contenus dans un ouvert de carte de $X$ qui est aussi un domaine de trivialisation de $E$, $\chi$ une fonction à valeurs positives, de classe $\ci$ sur $X$, à support compact contenu dans $V$ et identiquement égale à $1$ sur $U$ et si $f\in (L^p_{loc})^{r,q}(X,E)$, $q\geq 1$, après avoir identifié $U$ et $V$ avec des ouverts de $\cb^n$ par un choix de coordonnées, on obtient, pour tout $z\in V$, par la formule de Bochner-Martinelli-Koppelman
$$(-1)^{r+q} (\chi f)(z)=\opa\int_{\zeta\in V} (\chi f)(\zeta)\wedge B(z,\zeta)-\int_{\zeta\in V} \opa(\chi f)(\zeta)\wedge B(z,\zeta),$$
où $B(z,\zeta)$ désigne le noyau de Bochner-Martinelli. Lorsque la forme différentielle $f$ est $\opa$-fermée alors
$$(-1)^{r+q} (\chi f)(z)=\opa\int_{\zeta\in V} (\chi f)(\zeta)\wedge B(z,\zeta)-\int_{\zeta\in V} \opa\chi(\zeta)\wedge f(\zeta)\wedge B(z,\zeta).$$
Rappelons que le noyau de Bochner-Martinelli vérifie $B=\tau^*k_{BM}$, où $k_{BM}$ est une forme différentielle singulière à l'origine, à coefficients localement intégrables sur $\cb^n$ et $\tau$ l'application qui à $(z,\zeta)$ associe $\zeta-z$ (cf. \cite{Lalivre}, chapitre III). Par conséquent, grâce aux propriétés de la convolution, la forme différentielle $\wt B (\chi f)(z)=\int_{\zeta\in V} (\chi f)(\zeta)\wedge B(z,\zeta)$ est dans $(L^p_{loc})^{r,q-1}(X,E)$.

La forme différentielle $\int_{\zeta\in V} \opa\chi(\zeta)\wedge f(\zeta)\wedge B(z,\zeta)$ est $\opa$-fermée et de classe $\ci$ sur $U$, car les singularités de $B$ sont concentrées sur la diagonale et le support de $\opa\chi$ contenu dans $V\setminus U$. Il résulte du lemme de Dolbeault pour les formes de classe $\ci$ qu'il existe un voisinage $U'\subset U$ de $x$ et une forme différentielle $g$ de classe $\ci$ sur $U'$ telle que $\int_{\zeta\in V} \opa\chi(\zeta)\wedge f(\zeta)\wedge B(z,\zeta)=\opa g$ sur $U'$. Par conséquent on a pour tout $z\in U'$
$$(-1)^{r+q}f(z)= \opa \big(\wt B (\chi f)(z)-g\big)$$
et la forme différentielle $\wt B (\chi f)(z)-g$ est à coefficients dans $L^p_{loc}$.

Le complexe $((L^p_{loc})^{r,\bullet}(X,E),\opa)$ est donc une résolution du faisceau $\Omega^r_E$ et les groupes de cohomologie $H_{L^P_{loc}}^{r,q}(X,E)$ de ce complexe sont isomorphes aux groupes de cohomologie du faisceau  $\Omega^p_E$,  i.e.
$$H^q(X,\Omega^r_E)\sim H_{L^P_{loc}}^{r,q}(X,E).$$

On en déduit aisément la proposition suivante~:

\begin{prop}\label{isom}
Soient $X$ une variété analytique complexe de dimension complexe $n$ et $E$ un fibré vectoriel holomorphe sur $X$. Pour tout $p\geq 1$ et $0\leq r,q\leq n$, les applications naturelles
$$\Phi~:~H^{r,q}(X,E)\to H_{L^P_{loc}}^{r,q}(X,E)\quad {\rm et }\quad \Psi~:~H_{L^P_{loc}}^{r,q}(X,E)\to H_{cour}^{r,q}(X,E)$$
sont des isomorphismes, appelés isomorphismes de Dolbeault. En particulier

(i) Pour toute $f\in (L^p_{loc})^{r,q}(X,E)$, $q\geq 1$, telle que $\opa f=0$, il existe $g\in (L^p_{loc})^{r,q-1}(X,E)$ telle que $f-\opa g\in\ec^{r,q}(X,E)$.

(ii) Si $f\in (L^p_{loc})^{r,q}(X,E)$, $q\geq 1$, vérifie $f=\opa S$ pour un courant $S\in\dc'^{r,q-1}(X,E)$, il existe $g\in (L^p_{loc})^{r,q-1}(X,E)$ telle que $f=\opa g$ sur $X$.

\end{prop}

Notons que si la variété $X$ est compacte l'espace des formes à coefficients dans $L^p_{loc}(X,E)$ coïncide avec l'espace des formes à coefficients dans $L^p(X,E)$ et par conséquent il résulte de la Proposition \ref{isom} que
$$H^q(X,\Omega^r_E)\sim H^{r,q}(X,E)\sim H_{L^P}^{r,q}(X,E)\sim H_{cour}^{r,q}(X,E).$$

Le point (i) de la Proposition \ref{isom} peut être précisé de la manière suivante~:

\begin{prop}\label{supp}
Soient $X$ une variété analytique complexe de dimension complexe $n$, $E$ un fibré vectoriel holomorphe sur $X$ et $p\geq 1$. Pour toute $f\in (L^p_{loc})^{r,q}(X,E)$, $q\geq 1$, telle que $\opa f=0$, et tout voisinage $U$ du support de $f$, il existe $g\in (L^p_{loc})^{r,q-1}(X,E)$ à support dans $U$ telle que $f-\opa g\in\ec^{r,q}(X,E)$.
\end{prop}
\begin{proof}[Démonstration]
Choisissons un voisinage $V$ du support de $f$ tel que $\ol V\subset U$ et soient $\chi_0$ et $\chi_1$ des fonctions de classe $\ci$ sur $X$ telles que $\chi_0=1$ sur un voisinage de $X\setminus V$ et $\chi_0=0$ au voisinage du support de $f$, $\chi_1=1$ au voisinage de $X\setminus U$ et $\chi_1=0$ sur un voisinage de $\ol V$.

Puisque $f$ est $\opa$-fermée, il résulte du (i) de la Proposition \ref{isom} qu'il existe une forme $g_0\in (L^p_{loc})^{r,q-1}(X,E)$ telle que $f-\opa g_0=u\in\ec^{r,q}(X,E)$. Alors $u=-\opa g_0$ sur $X\setminus {\rm supp}f$ et comme le morphisme $\Phi$ de la Proposition \ref{isom} est injectif, il existe une forme  $v$ de classe $\ci$ sur $X\setminus {\rm supp}f$ telle que $u=\opa v$ sur $X\setminus {\rm supp}f$. Sur $X\setminus \ol V$ on a alors $\opa(g_0+\chi_0 v)=0$ et l'assertion (i) de la Proposition \ref{isom} appliquée à la variété $X\setminus \ol V$ implique alors l'existence d'une forme $g_1\in (L^p_{loc})^{r,q-1}(X\setminus \ol V,E)$ telle que $g_0+\chi_0 v-\opa g_1=w\in\ec^{r,q}(X\setminus \ol V,E)$. La forme $g=g_0+\chi_0 v-\chi_1 w-\opa (\chi_1 g_1)$ est à coefficients dans $L^p_{loc}$ et à support dans $U$, de plus $$f-\opa g=f-\opa (g_0+\chi_0 v-\chi_1 w)=u-\opa (\chi_0 v-\chi_1 w)$$ est de classe $\ci$ sur $X$.
\end{proof}

Considérons les sous espaces $(L^p_c)^{r,q}(X,E)$ de $(L^p_{loc})^{r,q}(X,E)$, $0\leq r,q\leq n$, des $(r,q)$-formes à coefficients dans $L^p$ et à support compact, on peut leur associés le complexe $((L^p_c)^{r,\bullet}(X,E),\opa)$. Les groupes de cohomologie de ce complexe sont définis par
$$H_{c,L^p}^{r,q}(X,E)=\frac{(L^p_c)^{r,q}(X,E)\cap\ker\opa}{\opa(L^p_c)^{r,q-1}(X,E)}.$$

\begin{cor}\label{compact}
Soient $X$ une variété analytique complexe, connexe, non compacte, de dimension complexe $n$ et $E$ un fibré vectoriel holomorphe sur $X$. Pour tout $p\geq 1$ et $0\leq r,q\leq n$, l'application naturelle
$$\Phi_c~:~H_c^{r,q}(X,E)\to H_{c,L^P}^{r,q}(X,E)$$
est surjective.
\end{cor}
\begin{proof}[Démonstration]
Si $q=0$, $H_c^{r,0}(X,E)= H_{c,L^P}^{r,0}(X,E)=0$, car $X$ est connexe, et le corollaire est vrai, il reste alors à
prouver que pour toute $f\in (L^p_c)^{r,q}(X,E)$, $q\geq 1$, telle que $\opa f=0$, il existe $g\in (L^p_c)^{r,q-1}(X,E)$ telle que $f-\opa g\in\dc^{r,q}(X,E)$. Soient $f\in (L^p_c)^{r,q}(X,E)$, $q\geq 1$, telle que $\opa f=0$, et $U\subset\subset X$ un voisinage relativement compact du support de $f$. Appliquons la Proposition \ref{supp} à $f$ et $U$, il existe donc $g\in (L^p_{loc})^{r,q-1}(X,E)$ à support dans $U$ telle que $f-\opa g\in\ec^{r,q}(X,E)$. Mais puisque $U$ est relativement compact dans $X$, le support de $g$ est compact et $g\in (L^p_c)^{r,q-1}(X,E)$, de plus le support de $f-\opa g$ est également contenu dans $U$, donc compact et $f-\opa g\in\dc^{r,q}(X,E)$.
\end{proof}

Le Corollaire \ref{compact} permet de répondre à la question naturelle suivante~: si la variété $X$ satisfait $H_c^{r,q}(X,E)=0$, $q\geq 1$, et si $f\in (L^p)^{r,q}(X,E)$ est une forme $\opa$-fermée à support compact dans $X$, existe-t-il une forme $g\in (L^p)^{r,q}(X,E)$ à support compact dans $X$ telle que $\opa f=g$ sur $X$ ? En effet, si $H_c^{r,q}(X,E)=0$, la surjectivité de l'application $\Phi_c$ implique que $H_{c,L^P}^{r,q}(X,E)=0$ et la réponse est donc positive par définition des groupes de cohomologie $H_{c,L^P}^{r,q}(X,E)$.

Cette question, originellement posée par G. Tomassini, a été étudiée par E. Amar et S. Mongodi \cite{AmMo} lorsque $X=\cb^n$ et par E. Amar \cite{Am} lorsque $X$ est une variété de Stein.

Nous pouvons également préciser l'assertion (ii) de la Proposition \ref{isom} en termes de régularité et de support.

\begin{prop}\label{reg}
Soient $X$ une variété analytique complexe de dimension complexe $n$, $E$ un fibré vectoriel holomorphe sur $X$ et $p\geq 1$.
Si $f\in (L^p_{loc})^{r,q}(X,E)$, $q\geq 1$, vérifie $f=\opa S$ pour un courant $S\in\dc'^{r,q-1}(X,E)$, alors pour tout voisinage $U$ du support de $S$ il existe $g\in (W^{1,p}_{loc})^{r,q-1}(X,E)$ à support dans $U$ telle que $f=\opa g$ sur $X$.
\end{prop}
\begin{proof}[Démonstration]
Soient $f\in (L^p_{loc})^{r,q}(X,E)$, $q\geq 1$, vérifiant $f=\opa S$ pour un courant $S\in\dc'^{r,q-1}(X,E)$ et $U$ un voisinage du support de $S$. Grâce à régularité intérieure de la solution canonique de l'équation de Cauchy-Riemann (cf. Théorème 4 (b) dans \cite{BGS}, car à l'intérieur du domaine tous les champs de vecteurs sont admissibles), il existe $g_0\in (W^{1,p}_{loc})^{r,q-1}(X,E)$ telle que $f=\opa g_0$ sur $X$. Le courant $S-g_0$ étant $\opa$-fermé, la surjectivité de l'application naturelle $$H^{r,q}(X,E)\to H_{cour}^{r,q}(X,E)$$ implique qu'il existe un courant $T$ sur $X$ et une forme $u$ $\opa$-fermée, de classe $\ci$ sur $X$ telle que $S-g_0+\opa T=u$. On a alors $\opa T=u+g_0$ sur $X\setminus {\rm supp}~S$. Une fois encore la régularité intérieure implique qu'il existe $g_1\in (W^{1,p}_{loc})^{r,q-1}(X\setminus {\rm supp}~S,E)$ telle que $u+g_0=\opa g_1$ sur $X\setminus {\rm supp}~S$. Choisissons une fonction $\chi$ de classe $\ci$ sur $X$ telle que $\chi=0$ au voisinage du support de $S$ et $\chi=1$ sur un voisinage de $X\setminus U$. Alors la forme $g=g_0+u-\opa (\chi g_1)$ est dans $(W^{1,p}_{loc})^{r,q-1}(X,E)$, elle vérifie  $\opa g=\opa (g_0+u)=\opa S=f$ et son support est contenu dans $U$.
\end{proof}

La Proposition \ref{reg} permet d'affirmer que dans une variété analytique complexe non compacte quelconque, si on sait qu'une forme $f\in (L^p_c)^{r,q}(X,E)$, $q\geq 1$, est exacte au sens des courants à support compact, alors il existe $g\in (W^{1,p}_c)^{r,q-1}(X,E)$ telle que $f=\opa g$ sur $X$. De plus on peut choisir $g$ avec un support arbitrairement proche de celui de la solution au sens des courants.  Notons qu'une condition nécessaire sur $f$ pour être exacte au sens des courants à support compact est d'être orthogonale aux formes $\opa$-fermées de classe $\ci$ sur $X$, plus précisément $f$ doit satisfaire
$\int_X f\wedge\varphi=0$ pour toute forme $\varphi\in\ec^{n-r,n-q}(X,E^*)$ vérifiant $\opa\varphi=0$.

\begin{cor}\label{compact2}
Soient $X$ une variété analytique complexe, connexe, non compacte, de dimension complexe $n$ et $E$ un fibré vectoriel holomorphe sur $X$. Pour tout $p\geq 1$ et $0\leq r,q\leq n$, l'application naturelle
$$\Psi_c~:~H_{c,L^P}^{r,q}(X,E)\to H_{c,cour}^{r,q}(X,E)$$
est injective.
\end{cor}
\begin{proof}[Démonstration]
Si $q=0$, $H_{c,cour}^{r,0}(X,E)= H_{c,L^P}^{r,0}(X,E)=0$, car $X$ est connexe, et le corollaire est vrai, il reste alors à
prouver que pour toute $f\in (L^p_c)^{r,q}(X,E)$, $q\geq 1$, telle que $f=\opa S$ pour un courant $S\in\dc'^{r,q-1}(X,E)$ à support compact, il existe $g\in (L^p_c)^{r,q-1}(X,E)$ telle que $f=\opa g$ sur $X$. C'est une conséquence directe de la Proposition \ref{reg} en prenant $U$ un voisinage relativement compact du support de $S$.
\end{proof}

Nous allons préciser la réponse que l'on peut apporter à la question de G. Tomassini par un contrôle du support de la solution en fonction du support de la donnée lorsque certaines conditions géométriques sont réalisées.

Il résulte du Corollaire \ref{compact2} et de la Proposition \ref{reg} que si $D$ est un voisinage du support de la forme différentielle $\opa$-fermée $f\in (L^p_c)^{r,q}(X,E)$, qui vérifie $H_{c,cour}^{r,q}(D,E)=0$, il existe $g\in (W^{1,p}_c)^{r,q-1}(X,E)$ telle que $f=\opa g$ sur $X$ et ${\rm supp}~g\subset D$.

\begin{thm}
Soient $X$ une variété  analytique complexe, non compacte, de dimension complexe $n$ et $E$ un fibré vectoriel holomorphe sur $X$. Si $f\in (L^p_c)^{r,q}(X,E)$, $0\leq r\leq n$, $q\geq 1$, vérifie $\opa f=0$, si $1\leq q\leq n-1$, et $\int_X f\wedge\varphi=0$ pour toute $\varphi\in\ec^{n-r,0}(X,E)$ telle que $\opa\varphi=0$, si $q=n$, et s'il existe un ouvert de Stein $D$ tel que ${\rm supp}~f\subset D\subset\subset X$, alors il existe $g\in (W^{1,p}_c)^{r,q-1}(X,E)$ telle que ${\rm supp}~g\subset D$ et $\opa g=f$.
\end{thm}
\begin{proof}[Démonstration]
Puisque $D$ est un ouvert de Stein $H^{r,q}(D,E^*)=0$ pour tout $0\leq r\leq n$ et $1\leq q\leq n$, il résulte de la dualité de Serre que      $H_{c,cour}^{r,q}(D,E)=0$ pour tout $0\leq r\leq n$ et $1\leq q\leq n-1$, et que si un $(r,n)$-courant $T$ à support compact dans $D$ vérifie $<T,\varphi>=0$ pour toute $\varphi\in\ec^{n-r,0}(X,E)$ telle que $\opa\varphi=0$, alors il existe un $(r,n-1)$-courant $S$ à support compact dans $D$ tel que $\opa S=T$. Il suffit alors d'appliquer la Proposition \ref{reg} au domaine $D$.
\end{proof}

Plus généralement si $D$ est seulement $s$-complet au sens d'Andreotti-Grauert (i.e. $D$ possède une fonction d'exhaustion de classe $\ci$, dont la forme de Levi possède au moins $n-s+1$ valeurs propres strictement positives), alors $H^{r,q}(D,E^*)=0$ pour tout $0\leq r\leq n$ et $s\leq q\leq n$ et par conséquent $H_{c,cour}^{r,q}(D,E)=0$ pour tout $0\leq r\leq n$ et $1\leq q\leq n-s$. On pourra donc résoudre l'équation de Cauchy-Riemann dans $L^p$ avec support dans $D$ pour une donnée de bidegré $(r,q)$ à support dans $D$ à condition que $1\leq q\leq n-s$. Ou encore si $D=\Omega\setminus\ol{\Omega^*}$, où $\Omega$ est $s$-complet et $\ol{\Omega^*}$ possède une base de voisinage $s^*$-complets, mais alors seulement pour $s^*+1\leq q\leq n-s$. En effet si $s^*+1\leq q\leq n-s$ et si $\ol{\Omega^*}$ possède une base de voisinage $s^*$-complets, nous allons pouvoir corriger toute solution $g$ à support compact dans $\Omega$ et $L^p$ dans $X$ de l'équation $\opa g=f$, lorsque $f\in L^p_{r,q}(X,E)$ est à support compact dans $D=\Omega\setminus\ol{\Omega^*}$ pour obtenir une solution $h$ à support compact dans $D$ et $L^p$ dans $X$. Puisque $f$ est à support compact dans $D$, la forme différentielle $g$ est $\opa$-fermée au voisinage de $\ol{\Omega^*}$ et par conséquent il existe un $(r,q-1)$-forme u définie et $L^p_{loc}$ au voisinage de $\ol{\Omega^*}$ telle que $g=\opa u$ ($\ol{\Omega^*}$ possède une base de voisinage $s^*$-complets). Il suffit alors de poser $h=g-\opa(\chi u)$, où $\chi$ est une fonction de classe $\ci$ sur $X$ à support compact dans le domaine de définition de $u$ et égale à $1$ au voisinage de $\ol{\Omega^*}$.

En degré $1$ nous avons le résultat plus précis suivant~:

\begin{prop}\label{degre1}
Soient $X$ une variété analytique complexe, connexe, non compacte, de dimension complexe $n$, $D$ un domaine relativement compact de $X$ et $E$ un fibré vectoriel holomorphe sur $X$. On suppose que, pour $0\leq r\leq n$, $H_c^{r,1}(X,E)=0$ (ce qui est satisfait par exemple si $X$ est $(n-1)$-complète au sens d'Andreotti-Grauert et en particulier pour les variétés de Stein de dimension $n\geq 2$) et que $X\setminus D$ est connexe, alors pour toute $(r,1)$-forme $f$ à coefficients dans $L^p(X,E)$ à support dans $\ol D$, il existe une unique $(r,0)$-forme $g$ à coefficients dans $W^{1,p}(X,E)$ et à support dans $\ol D$ telle que $\opa g=f$.
\end{prop}
\begin{proof}[Démonstration]
Il résulte du Corollaire \ref{compact}, qu'il existe  une $(r,0)$-forme $g$ à coefficients dans $L^p(X)$ et à support compact dans $X$ telle que $\opa g=f$. La forme $g$ s'annule donc sur un ouvert de $X\setminus D$. De plus $g$ est une $(r,0)$-forme holomorphe sur l'ouvert $X\subset{\rm supp}~f$ qui contient $X\setminus D$ et par prolongement analytique, puisque $X\setminus D$ est connexe, elle s'annule identiquement sur $X\setminus D$. Le support de $g$ est donc contenu dans $\ol D$.
Soit $h$ une autre solution, alors $g-h$ est une fonction holomorphe à support compact dans $X$ qui est donc nulle par prolongement analytique puisque $X$ est connexe. Par conséquent $g=h$ et de plus $g\in W^{1,p}(X,E)$ par la Proposition \ref{reg}.
\end{proof}

Nous retrouvons ainsi le contrôle du support proposé par E. Amar \cite{Am} en degré $1$, lorsque $X$ est une variété de Stein.
\medskip

Nous pouvons déduire de la Proposition \ref{degre1} une version $L^p$ du phénomène de Hartogs-Bochner-Severi (voir les chapitre IV et V de \cite{Lalivre} et les références associées pour l'étude de ce phénomène). Un version $L^2$ est donnée dans le dernier paragraphe de \cite{ChSh}.

\begin{thm}\label{hartogs}
Soient $X$ une variété analytique complexe, connexe, non compacte, de dimension complexe $n$, $D$ un domaine relativement compact de $X$ à bord Lipschitz et $E$ un fibré vectoriel holomorphe sur $X$. On suppose que, pour $0\leq r\leq n$, $H_c^{r,1}(X,E)=0$ (ce qui est satisfait par exemple si $X$ est $(n-1)$-complète au sens d'Andreotti-Grauert et en particulier par les variétés de Stein de dimension $n\geq 2$) et que $X\setminus D$ est connexe. Soit $f\in W^{1-\frac{1}{p}}_{r,0}(\pa D,E)$ telle que
$$\int_{\pa D} f\wedge\opa\varphi=0~{\rm pour~toute~} \varphi\in \ec^{n-r,n-1}(X,E^*).$$
Alors il existe une section $F$ de $E$ holomorphe sur $D$ et $W^{1,p}$ sur $X$ telle que la trace de $F$ sur $\pa D$ coïncide avec $f$.
\end{thm}
\begin{proof}[Démonstration]
Notons $\wt f$ une extension de $f$ contenue dans $W^{1,p}(X,E)$ (une telle extension existe lorsque le bord de $D$ est Lipschitz d'après le Théorème 1.5.1.3 de \cite{GrSobolev}) et $(\opa\wt f)_0$ le prolongement de $\opa\wt f_{|_D}$ par $0$ à $X$ . Nous allons montrer que la $(0,1)$-forme $(\opa\wt f)_0$ est $\opa$-fermée au sens des courants. Considérons une suite $(\wt f_\nu)_{\nu\in\nb}$ de $(0,1)$-formes de classe $\ci$ sur $X$ qui converge vers $\wt f$ sur $X$ dans $W^{1,p}$, alors, pour toute $\varphi\in \ec^{n-r,n-1}(X,E^*)$,
\begin{equation*}
\int_D\opa\wt f\wedge\opa\varphi=\lim_{\nu\to +\infty}\int_D\opa\wt f_\nu\wedge\opa\varphi=\lim_{\nu\to +\infty}\int_{\pa D}\wt f_\nu\wedge\opa\varphi=\int_{\pa D} f\wedge\opa\varphi=0.
\end{equation*}
D'après la Proposition \ref{degre1}, il existe une section $g\in W^{1,p}(X,E)$ à support dans $\ol D$ telle que $\opa g=(\opa\wt f)_0$. Puisque $g\in W^{1,p}(X,E)$ la trace de $g$ sur $\pa D$ existe et est nulle et $F=\wt f-g$ convient.
\end{proof} 

\section{Complexes de Cauchy-Riemann $L^p$ et dualité}\label{s2}

Dans cette section $X$ désigne une variété analytique complexe de dimension complexe $n$, $D$ un domaine relativement compact à bord rectifiable contenu dans $X$ et $E$ un fibré holomorphe de rang $l$ sur $X$. On munit $X$ d'une métrique hermitienne localement équivalente à la métrique usuelle de $\cb^n$ telle que le complexifié de l'espace tangent à $X$ vérifie $\cb T(X)=T^{1,0}(X)\oplus T^{0,1}(X)$ et que $T^{1,0}(X)\perp T^{0,1}(X)$. On note $dV$ la forme volume associée à cette métrique.

Pour tout entier $p\geq 1$, on définit l'espace $L^p(D)$ comme l'espace des fonctions $f$ sur $D$ à valeurs complexes telles que $|f|^p$ est intégrable sur $D$. Puisque $D\subset\subset X$, on a
$$\dc(D)\subset\ec(X)_{|_D}\subset L^p(D),$$
où $\ec(X)_{|_D}$ est l'espace des restrictions à $D$ des fonctions de classe $\ci$ sur $X$ et $\dc(D)$ le sous-espace des fonctions de classe $\ci$ à support compact dans $D$.
L'espace $L^p(D)$ est un espace de Banach, c'est le complété de $\dc(D)$ pour la norme $\|.\|_p$ définie par
$$\|f\|_p=(\int_D |f|^pdV)^\frac{1}{p}.$$

\subsection{Complexes de Cauchy-Riemann $L^p$}\label{complexe}

Pour $0\leq r,q\leq n$, $\ec^{r,q}(X,E)$ désigne l'espace des $(r,q)$-formes différentielles de classe $\ci$ sur $X$ à valeurs dans $E$ et $\dc^{r,q}(D,E)$ le sous-espace de $\ec^{r,q}(X,E)$ des formes à support compact dans $D$.
Soit $(U_i)_{i\in I}$ un recouvrement fini de $\ol D$ par des ouverts de cartes de $X$ qui sont également des ouverts de trivialisation de $E$ et $(\chi_i)_{i\in I}$ une partition de l'unité sur $D$ relative au recouvrement $(U_i)_{i\in I}$. On définit $L^p_{r,q}(D,E)$ comme l'espace des $(r,q)$-formes différentielles $f$ sur $D$ à valeurs dans $E$ telles que pour tout $i\in I$, dans des coordonnées locales, la forme $\chi_i f$ soit une forme à coefficients dans $L^p(D\cap U_i, \cb^l)$. Pour chaque choix de trivialisation et de coordonnées sur $U_i$ on peut définir  $\|\chi_i f\|_p$ comme la somme des normes $\|.\|_p$ de ses coefficients et on pose alors $\|f\|_p=\sum_{i\in I}\|\chi_i f\|_p$. Un changement de coordonnées ou de trivialisation définit une norme équivalente. L'espace $L^p_{r,q}(D,E)$ possède donc une structure d'espace de Banach et c'est le complété de $\dc^{r,q}(D,E)$ pour la norme $\|.\|_p$.

Pour $r$ fixé, $0\leq r\leq n$, nous allons associer aux espaces $L^p_{r,q}(D,E)$, $0\leq q\leq n$, des complexes d'opérateurs non bornés dont la restriction aux espaces $\dc^{r,q}(D,E)$, $0\leq q\leq n$, est le complexe de Cauchy-Riemann classique $(\dc^{r,\bullet}(D,E),\opa)$.
Pour cela nous devons définir des opérateurs de $L^p_{r,q}(D,E)$ dans $L^p_{r,q+1}(D,E)$ dont la restriction à $\dc^{r,q}(D,E)$ coïncide avec l'opérateur de Cauchy-Riemann $\opa$ associé à la structure complexe de $X$.

Définissons $\opa_c~:~L^p_{r,q}(D,E)\to L^p_{r,q+1}(D,E)$ comme l'extension fermée minimale de $\opa_{|_\dc}=\opa_{|_{\dc^{r,q}(D,E)}}$, c'est par définition un opérateur fermé et $f\in Dom(\opa_c)$ si et seulement si il existe une suite $(f_\nu)_{\nu\in\nb}$ d'éléments de $\dc^{r,q}(D,E)$ et $g\in\dc^{r,q+1}(D,E)$ tels que $(f_\nu)_{\nu\in\nb}$ converge vers $f$ dans $L^p_{r,q}(D,E)$ et $(\opa f_\nu)_{\nu\in\nb}$  converge vers $g$ dans $L^p_{r,q+1}(D,E)$, on a alors $g=\opa_c f$.

On peut aussi considérer $\opa_s~:~L^p_{r,q}(D,E)\to L^p_{r,q+1}(D,E)$, l'extension fermée minimale de $\opa_{|_\ec}=\opa_{|_{\ec^{r,q}(X,E)_{|_D}}}$, c'est également un opérateur fermé et $f\in Dom(\opa_s)$ si et seulement si il existe une suite $(f_\nu)_{\nu\in\nb}$ d'éléments de $\ec^{r,q}(X,E)$ et $g\in\ec^{r,q+1}(X,E)$ tels que $(f_\nu)_{\nu\in\nb}$ converge vers $f$ dans $L^p_{r,q}(D,E)$ et $(\opa f_\nu)_{\nu\in\nb}$  converge vers $g$ dans $L^p_{r,q+1}(D,E)$, on a alors $g=\opa_s f$.

L'opérateur $\opa$ s'étend à $L^p_{r,q}(D,E)$ au sens des courants, on peut alors considérer les opérateurs $\opa_{\wt c}~:~L^p_{r,q}(D,E)\to L^p_{r,q+1}(D,E)$ et
$\opa~:~L^p_{r,q}(D,E)\to L^p_{r,q+1}(D,E)$ qui coïncident avec l'opérateur $\opa$ au sens des courants et dont les domaines de définitions sont respectivement
$$Dom(\opa_{\wt c})=\{f\in L^p_{r,q}(X,E)~|~{\rm supp}~f\subset\ol D, \opa f\in L^p_{r,q}(X,E)\}$$
et
$$Dom(\opa)=\{f\in L^p_{r,q}(D,E)~|~ \opa f\in L^p_{r,q}(D,E)\}.$$

\begin{defin}
Un \emph{complexe cohomologique d'espaces de Banach} est un couple $(E^\bullet,d)$, où $E^\bullet=(E^q)_{q\in\zb}$ est une suite d'espaces de Banach et $d=(d^q)_{q\in\zb}$ une suite d'opérateurs linéaires $d^q$ fermés de $E^q$ dans $E^{q+1}$ qui vérifient $d^{q+1}\circ d^q=0$.

Un \emph{complexe homologique d'espaces de Banach} est un couple $(E_\bullet,d)$, où $E_\bullet=(E_q)_{q\in\zb}$ est une suite d'espaces de Banach et $d=(d_q)_{q\in\zb}$ une suite d'opérateurs linéaires $d_q$ fermés de $E_q$ dans $E_{q+1}$ qui vérifient $d_{q}\circ d_{q+1}=0$.
\end{defin}

A tout complexe cohomologique on associe les \emph{groupes de cohomologie} $(H^q(E^\bullet))_{q\in\zb}$ définis par
$$H^q(E^\bullet)=\ker d^q/\im d^{q-1}$$
et on les munit de la topologie quotient. De même à tout complexe homologique on associe les \emph{groupes d'homologie} $(H_q(E_\bullet))_{q\in\zb}$ définis par
$$H_q(E_\bullet)=\ker d_{q-1}/\im d_q$$
et on les munit également de la topologie quotient.

En remarquant que les opérateurs $\opa$ et $\opa_s$ satisfont $\opa\circ\opa=0$ et $\opa_s\circ\opa_s=0$, on peut considérer, pour tout $0\leq r\leq n$, les complexes $(L^p_{r,\bullet}(D,E),\opa)$ et $(L^p_{r,\bullet}(D,E),\opa_s)$. On note $H^{r,q}_{L^p}(D,E)$ et $H^{r,q}_{L^p,s}(D,E)$ leurs groupes de cohomologie respectifs

Observons que lorsque le bord de $D$ est assez régulier, Lipschitz par exemple, il résulte du lemme de Friedrich (cf. lemme 4.3.2 dans \cite{ChSh} et lemme 2.4 dans \cite{LaShdualiteL2}) que les opérateurs $\opa$ et $\opa_s$ coïncident ainsi que les opérateurs $\opa_c$ et $\opa_{\wt c}$.

\begin{prop}\label{lip}
Supposons que $D\subset\subset X$ est un domaine à bord Lipschitz alors

(i) Une forme différentielle $f\in L^p_{r,q}(D,E)$ appartient au domaine de définition de l'opérateur $\opa$ si et seulement si il existe une suite $(f_\nu)_{\nu\in\nb}$ d'éléments de $\ec^{r,q}(D,E)$ telle que les suites $(f_\nu)_{\nu\in\nb}$ et $(\opa f_\nu)_{\nu\in\nb}$ convergent respectivement vers $f$ et $\opa f$ en norme $\|.\|_p$.

(ii) Une forme différentielle $f\in L^p_{r,q}(D,E)$ appartient au domaine de définition de l'opérateur $\opa_c$ si et seulement si les formes $f_0$ et $\opa f_0$ sont toutes deux dans $L^p_*(X,E)$, si $f_0$ désigne la forme obtenue en prolongeant $f$ par $0$ sur $X\setminus D$.
\end{prop}

On déduit alors immédiatement du (i) de la Proposition \ref{lip} que, lorsque le bord de $D$ est Lipschitz, les groupes de cohomologie $H^{r,q}_{L^p}(D,E)$ et $H^{r,q}_{L^p,s}(D,E)$, $0\leq r,q \leq n$, sont égaux.

\subsection{Théorie d'Andreotti-Grauert $L^p$}

L'objet de cette section est d'étendre aux groupes de cohomologie $L^p$ la théorie classique d'Andreotti-Grauert. Pour cela nous allons associer des résultats classiques sur la résolution de l'équation de Cauchy-Riemann avec estimations $L^p$ et la méthode des bosses de Grauert.

\begin{defin}\label{s-conv}
Soient $X$ une variété analytique complexe de dimension complexe $n$, $D$ un domaine relativement compact dans $X$ et $s$ un entier, $0\leq s\leq n-1$. On dira que $D$ est un \emph{domaine complètement strictement $s$-convexe}, s'il existe une fonction $\rho$ de classe $\cc^2$ à valeurs réelles définie sur un voisinage $U$ de $\ol D$, dont la forme de Levi possède au moins $n-s+1$ valeurs propres strictement positives, telle que $d\rho(x)\neq 0$ pour tout $x\in\pa D$ et
$$D=\{x\in U~|~\rho(x)<0\}.$$
On dira que $D$ est \emph{strictement $s$-convexe} si la fonction $\rho$ est seulement définie au voisinage du bord de $D$.
\end{defin}

Remarquons que grâce au lemme de Morse, la fonction $\rho$ dans la Définition \ref{s-conv} peut, sans perte de généralité, être supposée sans point critique dégénéré.
\medskip

Il résulte des travaux de L. Ma \cite{Ma} et R. Beals, P. Greiner et N. Stanton \cite{BGS} que l'on peut résoudre l'équation de Cauchy-Riemann avec estimations Sobolev $L^p$ dans les domaines complètement strictement $s$-convexes de $\cb^n$.

\begin{thm}\label{homotopie}
Soient $D\subset\subset\cb^n$ un domaine complètement strictement $s$-convexe de $\cb^n$ à bord $\ci$ et $f$ une $(r,q)$-forme différentielle à coefficients dans $W^{s,p}(D)$. Pour $0\leq r\leq n$ et $q\geq s$, il existe un opérateur linéaire continu $T_q$ de $W^{s,p}_{r,q}(\ol D)$ dans $W^{s+1/2,p}_{r,q-1}(\ol D)$ tel que
$$f=\opa T_q f+T_{q+1}\opa f.$$
\end{thm}

Observons que, puisque l'injection de $W^{1/2,p}(\ol D)$ dans $L^p(D)$ est compacte, en globalisant à l'aide d'une partition de l'unité le Théorème \ref{homotopie} comme dans le chapitre 11 de \cite{HeLe2} on obtient
\begin{thm}\label{glob}
Soient $X$ une variété analytique complexe de dimension complexe $n$, $E$ un fibré vectoriel holomorphe sur $X$ et $D\subset\subset X$ un domaine strictement $s$-convexe à bord $\ci$ de $X$, $r$ et $q$ des entiers tels que $0\leq r\leq n$ et $q\geq s$. Alors

(i) il existe des opérateurs linéaires continus $T_q^r$ de $L^p_{r,q}(D)$ dans $L^p_{r,q-1}(D)$ et un opérateur $K^r_q$ compact de $L^p_{r,q}(D)\cap Dom(\opa)$ dans lui-même tels que si $f\in L^p_{r,q}(D)$vérifie $\opa f\in L^p_{r,q+1}(D)$
$$f=\opa T^r_q f+T^r_{q+1}\opa f+K^r_qf.$$

(ii) si $q\geq \max(1,s)$, $H^{r,q}_{L^p}(D,E)$ est de dimension finie et $\opa(L^p_{r,q-1}(D))$ est un sous espace fermé de $L^p_{r,q}(D)$.
\end{thm}

Nous pouvons maintenant mettre en {\oe}uvre la méthode des bosses de Grauert. Rappelons quelques définitions et quelques résultats géométriques donnés dans le chapitre 12 de \cite{HeLe2}.

\begin{defin}
On appelle \emph{élément d'extension strictement $s$-convexe} un couple ordonné $[\theta_1,\theta_2]$ d'ouverts de $X$ à bord $\ci$ tel que $\theta_1\subset\theta_2$ satisfaisant la condition suivante~:~il existe un ouvert pseudoconvexe $V$ contenu dans un domaine de trivialisation de $E$ contenant ${\ol\theta_2\setminus\theta_1}$ et des domaines strictement $s$-convexes $D_1$ et $D_2$ tels que $D_1\subset D_2$, $\theta_2=\theta_1\cup D_2$, $\theta_1\cap D_2=D_1$ et $(\ol{\theta_1\setminus D_2})\cap(\ol{\theta_2\setminus\theta_1})=\emptyset$ et une application biholomorphe $h$ définie sur un voisinage de $\ol V$ à valeurs dans $\cb^n$ telle que $h(D_j)$, $j=1,2$, soit un domaine complètement strictement $q$-convexe à bord $\ci$ de $\cb^n$.
\end{defin}

\begin{defin}
Soient $D\subset\subset \Omega\subset\subset X$ des ouverts de $X$ à bord $\ci$. On dira que $\Omega$ est une \emph{extension strictement $s$-convexe} de $D$, s'il existe un voisinage $U$ de $\ol\Omega\setminus D$ et une fonction $\rho$ de classe $\cc^2$ à valeurs réelles sur $U$ avec un ensemble de points critiques discret et dont la forme de Levi possède au moins $n-s+1$ valeurs propres strictement positives, telle que
$$D\cap U=\{x\in U~|~\rho(x)<0\}\quad {\rm et}\quad d\rho(x)\neq 0 \quad {\rm si}\quad x\in\pa D,$$
$$\Omega\cap U=\{x\in U~|~\rho(x)<1\}\quad {\rm et}\quad d\rho(x)\neq 0 \quad {\rm si}\quad x\in\pa \Omega.$$
\end{defin}

\begin{prop}\label{extgeom}
Soient $D\subset\subset \Omega\subset\subset X$ des ouverts de $X$ à bord $\ci$ tels que $\Omega$ soit une extension strictement $s$-convexe de $D$. Il existe une suite finie $\theta_0,\dots,\theta_N$ d'ouverts de $X$ tels que
$$D=\theta_0\subset\dots\subset\theta_N=\Omega$$
et que $[\theta_{j-1},\theta_j]$, $j=1,\dots,N$, soit un élément d'extension strictement $s$-convexe.
\end{prop}

Nous sommes maintenant en mesure de prouver l'invariance de la cohomologie $L^p$ par les extensions strictement $s$-convexes.

\begin{thm}\label{ext}
Soient $X$ une variété analytique complexe de dimension complexe $n$, $E$ un fibré vectoriel holomorphe sur $X$ et
$D\subset\subset \Omega\subset\subset X$ des ouverts de $X$ à bord $\ci$ tels que $\Omega$ soit une extension strictement $s$-convexe de $D$. Alors l'application restriction
$$H^{r,q}_{L^p}(\Omega,E)\to H^{r,q}_{L^p}(D,E),\quad 0\leq r\leq n,~\max(1,s)\leq q\leq n,$$
est un isomorphisme.
\end{thm}
\begin{proof}[Démonstration]
La Proposition \ref{extgeom} nous permet de réduire la démonstration au cas où le couple $(D,\Omega)$ est un élément d'extension strictement $s$-convexe $[\theta_1,\theta_2]$, où $\theta_1$ et $\theta_2$ sont tous deux strictement $s$-convexes.

Puisque par définition des éléments d'extension $(\ol{\theta_1\setminus D_2})\cap(\ol{\theta_2\setminus\theta_1})=\emptyset$, on peut trouver des voisinages $V'$ et $V''$ de $\ol{\theta_2\setminus\theta_1}$ tels que $V'\subset\subset V''\subset\subset V$ et $V''\cap (\ol{\theta_1\setminus D_2})=\emptyset$. Choisissons une fonction $\chi$ de classe $\ci$ dans $X$ telle que $\chi=1$ sur $V'$ et ${\rm supp}~\chi\subset V''$.

Prouvons tout d'abord que l'application restriction
$$H^{r,q}_{L^p}(\theta_2,E)\to H^{r,q}_{L^p}(\theta_1,E)$$
est surjective, si $s\leq q\leq n$.
Soit $f_1\in L^p_{r,q}(\theta_1)$ telle que $\opa f_1=0$, on cherche des formes différentielles $u_1\in L^p_{r,q-1}(\theta_1)$ et $f_2\in L^p_{r,q}(\theta_2)$ telles que $\opa f_2=0$ et $f_2=f_1-\opa u_1$ sur $\theta_1$. Comme $D_1$ est l'image par une application biholomorphe, définie au voisinage $\ol D_1$, d'un domaine complètement strictement $s$-convexe borné à bord $\ci$ de $\cb^n$, il existe $u\in  L^p_{r,q-1}(D_1)$ telle que $f_1=\opa u$ sur $D_1$, lorsque $q\geq \max(1,s)$ (cf. Théorème \ref{homotopie}). Posons
$$
u_1=\left\{
\begin{array}{cl}
0 &\quad {\rm sur}~\ol{\theta_1\setminus D_2}\\
\chi u &\quad {\rm sur}~\ol D_1
\end{array}
\right.
\quad {\rm et}\quad f_2=\left\{
\begin{array}{cl}
f_1-\opa u_1 &\quad {\rm sur}~\ol{\theta_1}\\
0 &\quad {\rm sur}~\ol{\theta_2\setminus\theta_1}
\end{array}
\right.
.$$
Les formes différentielles $u_1$ et $f_2$ satisfont les conditions demandées.

Considérons maintenant l'injectivité. Soit $f_2\in L^p_{r,q}(\theta_2)$ telle que $\opa f_2=0$ et $u_1\in L^p_{r,q-1}(\theta_1)$ telle que $f_2=\opa u_1$ sur $\theta_1$, on cherche une forme différentielle $u_2\in L^p_{r,q-1}(\theta_2)$ telle que $f_2=\opa u_2$ sur $\theta_2$. Comme $D_2$ est l'image par une application biholomorphe, définie au voisinage $\ol D_2$, d'un domaine complètement strictement $s$-convexe borné à bord $\ci$ de $\cb^n$, il existe $u\in  L^p_{r,q-1}(D_2)$ telle que $f_2=\opa u$ sur $D_2$, lorsque $q\geq \max(1,s)$, (cf. Théorème \ref{homotopie}) et par conséquent $\opa (u-u_1)=0$ sur $D_1$.

Si $q\geq \max(1,s)+1$, par le Théorème \ref{homotopie}, il existe $v\in  L^p_{r,q-2}(D_1)$ telle que $\opa v=u-u_1$ sur $D_1$. Alors $u_1-\opa(\chi v)=u$ sur $V'\cap\theta_1$. Posons
$$
u_2=\left\{
\begin{array}{cl}
u_1-\opa(\chi v) &\quad {\rm sur}~\theta_1\\
u &\quad {\rm sur}~V'\cap\theta_2
\end{array}
\right.
,$$
alors $u_2\in L^p_{r,q-1}(\theta_2)$ et $f_2=\opa u_2$ sur $\theta_2$.

Supposons que $q=\max(1,s)$, il résulte du Théorème \ref{glob} que $\opa(L^p_{r,q-1}(\theta_2))$ est un sous espace fermé de $L^p_{r,q}(\theta_2)$. Il suffit donc de construire une suite $(w_\nu)_{\nu\in\nb}$ de $(r,q-1)$-formes contenue dans $Dom(\opa)$ telle que la suite $(\opa w_\nu)_{\nu\in\nb}$ converge vers $f_2$ en norme $\|.\|_p$ sur $\theta_2$. La forme différentielle $u-u_1$ est à coefficients dans $L^p(D_1)$ et $\opa$-fermée sur $D_1$, donc $u-u_1\in Dom(\opa)$. Puisque $D_1$ est à bord $\ci$ il existe une suite $(v_\nu)_{\nu\in\nb}$ de formes à coefficients $\ci$ dans $V$ telle $\lim_{\nu\to+\infty} v_\nu=u-u_1$ et $\lim_{\nu\to+\infty} \opa v_\nu=0$ en norme $\|.\|_p$ sur $D_1$. Comme $D_1$ est l'image par une application biholomorphe, définie au voisinage $\ol D_1$, d'un domaine complètement strictement $s$-convexe borné à bord $\ci$ de $\cb^n$, la solution canonique $g_\nu$ de l'équation $\opa g=\opa v_\nu$ sur $D_1$ est de classe $\ci$ sur $\ol D_1$ et vérifie $\|g_\nu\|_p\leq C\|\opa v_\nu\|_p$ sur $D_1$, la constante $C$ étant indépendante de $\nu$. En posant $\wt v_\nu=v_\nu-g_\nu$, on obtient une suite $(\wt v_\nu)_{\nu\in\nb}$ de formes $\opa$-fermées sur $D_1$ à coefficients $\ci$ sur $\ol D_1$ qui converge vers $u-u_1$ en norme $\|.\|_p$ sur $D_1$. Il résulte du Corollaire 12.5 de \cite{HeLe2}, que, pour tout $\nu\in\nb$, il existe une $(r,q-1)$-forme continue $\wt w_\nu$ $\opa$-fermée sur $V$ telle que
$$\|\wt w_\nu-\wt v_\nu\|_{\infty,~\ol\theta_1\cap V}<\frac{1}{2^n}$$
et par conséquent si $\wt V$ est un ouvert qui vérifie $\ol D_2\subset\wt V$ et $V'\subset\subset V''\subset\subset\wt V\subset\subset V$
$$\|\wt w_\nu-\wt v_\nu\|_p<\wt C\frac{1}{2^n}\quad {\rm sur}~\ol\theta_1\cap \wt V.$$
On obtient ainsi une suite $(\wt w_\nu)_{\nu\in\nb}$ de formes $\opa$-fermée à coefficients $L^p$ dans $V''$ qui converge vers $u-u_1$ en norme $\|.\|_p$ sur $D_1$. On définit alors la suite $(w_\nu)_{\nu\in\nb}$ en posant $w_\nu=(1-\chi)u_1+\chi(u-\wt w_\nu)$. alors $\opa w_\nu=f_2+\opa\chi(u-u_1-\wt w_\nu)$ et la suite $(w_\nu)_{\nu\in\nb}$ satisfait bien les conditions demandées.
\end{proof}

On en déduit le corollaire suivant en utilisant la même stratégie que dans la preuve du Théorème 12.15 de \cite{HeLe2}~:
\begin{cor}\label{isomLp-Lploc}
Soient $X$ une variété analytique complexe de dimension complexe $n$, $E$ un fibré vectoriel holomorphe sur $X$ et $D\subset\subset X$ un domaine strictement $s$-convexe. Alors pour tout couple $(r,q)$ tel que $0\leq r\leq n$, et $\max(1,s)\leq q\leq n$,
$$H^{r,q}_{L^p}(D,E)\sim H^{r,q}_{L^p_{loc}}(D,E).$$
\end{cor}

Pour conclure nous prouvons un théorème d'annulation pour la cohomologie $L^p$~:
\begin{cor}\label{resolutionLp}
Soient $X$ une variété analytique complexe de dimension complexe $n$, $E$ un fibré vectoriel holomorphe sur $X$ et $D\subset\subset X$ un domaine complètement strictement $s$-convexe. Alors pour tout couple $(r,q)$ tel que $0\leq r\leq n$, et $\max(1,s)\leq q\leq n$,
$$H^{r,q}_{L^p}(D,E)=0.$$
Plus précisément il existe un opérateur linéaire continu $T$ de l'espace de Banach~$Z^{r,q}_{L^p}(D,E)$ des $(r,q)$-formes $\opa$-fermées à valeurs dans $E$ et à coefficients dans $L^p(D)$ dans l'espace de Banach $W^{1/2,p}_{r,q-1}(\ol D,E)$ tel que
$$\opa Tf=f~{\rm sur}~D~{\rm pour~toute}~f\in Z^{r,q}_{L^p}(D,E).$$
\end{cor}
\begin{proof}[Démonstration]
Le Lemme 8.5 de \cite{Lalivre} restant valable pour les fonctions $\cc^2$ dont la forme de Levi possède au moins $n-s+1$ valeurs propres strictement positives, on peut supposer que l'ensemble des points critiques de la fonction définissante $\rho$ de $D$ est discret. Notons $a=\min_{x\in D}\rho(x)$. Si $x\in D$ vérifie $\rho(x)=a$, c'est un point critique de $\rho$. L'ensemble des points critiques de $\rho$ étant discret, il n'existe qu'un nombre fini de tels point dans $D$. Par conséquent, pour $\varepsilon>0$ assez petit, l'ensemble $D_{a+\varepsilon}=\{x\in U~|~\rho(x)<a+\varepsilon\}$ est contenu dans un ouvert de trivialisation de $E$ biholomorphe à une réunion finie de domaines complètement strictement $s$-convexes à bord $\ci$ de $\cb^n$. La nullité de $H^{r,q}_{L^p}(D,E)$ résulte alors des Théorèmes \ref{homotopie} et \ref{ext}.

L'existence de l'opérateur $T$ est une conséquence du Théorème \ref{glob} et de la Proposition 5 de l'annexe C de \cite{Lalivre} (voir aussi l'appendice 2 de \cite{HeLe1}).
\end{proof}

\subsection{Dualité}\label{dualite}

Nous allons considérer dans cette section les complexes duaux associés aux complexes définis dans la section \ref{complexe}.

\begin{defin}
Le \emph{complexe dual} d'un complexe cohomologique $(E^\bullet,d)$ d'espaces de Banach est le complexe homologique $(E'_\bullet,d')$, avec $E'_\bullet=(E'_q)_{q\in\zb}$, où $E'_q$ est le dual fort de $E^q$ et $d'=(d'_q)_{q\in\zb}$, où $d'_q$ est l'opérateur transposé de l'opérateur $d^q$. Si les espaces $E^q$ sont réflexifs on parlera de \emph{paire} de complexes duaux.
\end{defin}

Pour tout entier $p$ tel que $1<p<+\infty$, on note $p'$ le nombre entier défini par la relation $\frac{1}{p}+\frac{1}{p'}=1$. Rappelons que si $f\in L^p(D)$ et $g\in L^{p'}(D)$ alors
$$\int_D|fg|~dV\leq \|f\|_p\|g\|_{p'}$$
et que les espaces  $L^{p}(D)$ et $L^{p'}(D)$ sont des espaces de Banach réflexifs duaux l'un de l'autre.
Si $E^*$ désigne le fibré dual du fibré $E$, il en résulte la même relation de dualité entre les espaces $L^p_{r,q}(D,E)$ et $L^{p'}_{n-r,n-q}(D,E^*)$.

Dans toute la suite nous supposerons que l'entier $p$ vérifie $1<p<+\infty$ et nous noterons $p'$ l'entier conjugué à $p$, c'est-à-dire tel que $\frac{1}{p}+\frac{1}{p'}=1$.

\begin{prop}\label{transp}
(i) L'opérateur transposé de l'opérateur $\opa_c$ de $L^p_{r,q}(D,E)$ dans $L^p_{r,q+1}(D,E)$ est l'opérateur $\opa$ de $L^{p'}_{n-r,n-q-1}(D,E^*)$ dans $L^{p'}_{n-r,n-q}(D,E^*)$.

(ii) L'opérateur $\opa$ de $L^p_{r,q}(D,E)$ dans $L^p_{r,q+1}(D,E)$ est fermé et son transposé est l'opérateur $\opa_c$ de $L^{p'}_{n-r,n-q-1}(D,E^*)$ dans $L^{p'}_{n-r,n-q}(D,E^*)$.

(iii) L'opérateur transposé de l'opérateur $\opa_s$ de $L^p_{r,q}(D,E)$ dans $L^p_{r,q+1}(D,E)$ est l'opérateur $\opa_{\wt c}$ de $L^{p'}_{n-r,n-q-1}(D,E^*)$ dans $L^{p'}_{n-r,n-q}(D,E^*)$.

(iv) L'opérateur $\opa_{\wt c}$ de $L^p_{r,q}(D,E)$ dans $L^p_{r,q+1}(D,E)$ est fermé et son transposé est l'opérateur $\opa_s$ de $L^{p'}_{n-r,n-q-1}(D,E^*)$ dans $L^{p'}_{n-r,n-q}(D,E^*)$.
\end{prop}
\begin{proof}[Démonstration]
(i) Puisque le domaine de définition de l'opérateur $\opa_c$ est dense dans $L^p_{r,q}(D,E)$, il résulte du Théorème II.2.11 de \cite{Gold} que $^t(\opa_c)=~^t(\opa_{|_\dc})$. Il suffit donc de prouver que $^t(\opa_{|_\dc})=\opa$. Soit $g\in L^{p'}_{n-r,n-q-1}(D,E^*)$, considérons l'application de $\dc^{r,q}(D,E)$ à valeurs dans $\cb$ définie par $\varphi\mapsto <g,\opa_{|_\dc}\varphi>$, où $<.,.>$ désigne le crochet de dualité au sens des distributions. Elle est continue en $\|.\|_p$ si et seulement si $g\in Dom(\opa)$.
En effet si $g\in Dom(\opa)$ alors $\opa g\in L^{p'}_{n-r,n-q}(D,E^*)$ et on a
$$|<g,\opa_{|_\dc}\varphi>|=|<\opa g,\varphi>|=|\int_D \opa g\wedge\varphi|\leq C\|\opa g\|_{p'}\|\varphi\|_p$$
et si $\varphi\mapsto <g,\opa_{|_\dc}\varphi>$ est continue en $\|.\|_p$, elle définit une forme différentielle $h\in L^{p'}_{n-r,n-q}(D,E^*)$ qui vérifie $h=\opa g$ dans $D$ au sens des distributions, donc $g\in Dom(\opa)$.

(ii) Nous venons de prouver que $^t(\opa_c)=\opa$, l'opérateur $\opa$ est donc fermé. De plus par réflexivité $^t(\opa)=\opa_c$.

(iii) Puisque le domaine de définition de l'opérateur $\opa_s$ est dense dans $L^p_{r,q}(D,E)$, il suffit, comme dans la preuve de (i), de montrer que $^t(\opa_{|_\ec})=\opa_{\wt c}$. Soit $g\in L^{p'}_{n-r,n-q-1}(X,E^*)$ à support dans $\ol D$, considérons l'application de $\ec^{r,q}(X,E)$ à valeurs dans $\cb$ définie par $\varphi\mapsto <g,\opa_{|_\ec}\varphi>=\int_{\ol D} g\wedge\opa_{|_\ec}\varphi$. Elle est continue en $\|.\|_p$ si et seulement si $g\in Dom(\opa_{\wt c})$. En effet si $g\in Dom(\opa_{\wt c})$, alors $g\in L^{p'}_{n-r,n-q-1}(X,E^*)$, $\opa g\in L^{p'}_{n-r,n-q}(X,E^*)$ et leur support est contenu dans $\ol D$, alors
$$|\int_{\ol D} g\wedge\opa_{|_\ec}\varphi|=|<g,\opa_{|_\ec}\varphi>|=|<\opa g,\varphi>|=|\int_{\ol D} \opa g\wedge\varphi|\leq C\|\opa g\|_{p'}\|\varphi_{|_D}\|_p$$ et si $\varphi\mapsto <g,\opa_{|_\ec}\varphi>=\int_{\ol D} g\wedge\opa_{|_\ec}\varphi$ est continue en $\|.\|_p$, elle définit une forme différentielle $h\in L^{p'}_{n-r,n-q}(X,E^*)\cap\ec'_{n-r,n-q}(X,E^*)$ dont le support est contenu dans $\ol D$ et qui vérifie $h=\opa g$ dans $X$ au sens des distributions, donc $g\in Dom(\opa_{\wt c})$.

(iv) Nous venons de prouver que $^t(\opa_s)=\opa_{\wt c}$, l'opérateur $\opa_{\wt c}$ est donc fermé. De plus par réflexivité $^t(\opa_{\wt c})=\opa_s$.
\end{proof}

Grâce à la Proposition \ref{transp}, on peut définir, pour tout $0\leq r\leq n$, deux paires de complexes duaux :
$$\big((L^p_{r,\bullet}(D,E),\opa);(L^{p'}_{n-r,\bullet}(D,E^*),\opa_c)\big)$$
et
$$\big((L^p_{r,\bullet}(D,E),\opa_s);(L^{p'}_{n-r,\bullet}(D,E^*),\opa_{\wt c})\big).$$
Si $E^q=L^p_{r,q}(D,E)$ alors $E'_q=L^{p'}_{n-r,n-q}(D,E^*)$ et on note $H^{n-r,n-q}_{c,L^{p'}}(D,E^*)$ et $H^{n-r,n-q}_{\wt c,L^{p'}}(D,E^*)$
les groupes de cohomologie associés aux complexes duaux. Observons que si le bord de $D$ est Lipschitz, il résulte de la Proposition \ref{lip} que
$$H^{n-r,n-q}_{c,L^{p'}}(D,E^*)=H^{n-r,n-q}_{\wt c,L^{p'}}(D,E^*).$$

Rappelons maintenant quelques résultats de dualité concernant les opérateurs non bornés (cf. \cite{Gold}).

Soient $E$ et $F$ deux espaces de Banach et $T$ un opérateur fermé de $E$ dans $F$. Si $E'$ et $F'$ désignent respectivement les espaces duaux forts de $E$ et $F$, on note $^tT$ l'opérateur de $F'$ dans $E'$ transposé de l'opérateur $T$. C'est un opérateur fermé. De plus
$$\ol{\im T}=(\ker~^tT)^\circ,$$
où $(\ker~^tT)^\circ$ désigne le polaire de $\ker~^tT$, c'est à dire l'ensemble des éléments $x\in F$ tels que $\sigma(x)=0$ pour tout $\sigma\in\ker~^tT$.

\begin{thm}\label{dual}
Supposons que $T$ est un opérateur fermé à domaine dense de l'espace de Banach $E$ dans l'espace de Banach $F$, alors les assertions suivantes sont équivalentes~:

(i) $\im T$ est fermée.

(ii) $\im~^tT$ est fermée.

(iii) $\im T=(\ker~^tT)^\circ$.

(iv) $\im~^tT=(\ker T)^\circ$.
\end{thm}

Dans le cas de complexes duaux cela se traduit par

\begin{thm}\label{complexdual}
Soit $\big((E^\bullet,d);(E'_\bullet,d')\big)$ une paire de complexes d'espaces de Banach réflexifs duaux. Si les opérateurs définissant les complexes sont fermés à domaine dense, alors, pour tout $q\in\zb$, $H^{q+1}(E^\bullet)$ est séparé si et seulement si $H_q(E'_\bullet)$ est séparé.
\end{thm}

On obtient également comme dans \cite{LaShdualiteL2}

\begin{prop}\label{nonsep}
Soient $(E^\bullet,d)$ et $(E'_\bullet,d')$ deux complexes d'espaces de Banach duaux. Supposons que $H_{q+1}(E'_\bullet)=0$, alors soit $H^{q+1}(E^\bullet)=0$, soit $H^{q+1}(E^\bullet)$ n'est pas séparé.
\end{prop}
\begin{proof}[Démonstration]
Remarquons que
$$\ol{\im d^q}=\{g\in E^{q+1}~|~<g,f>=0,\forall f\in\ker d'_q\}\subset\ker d^{q+1}.$$
Cette inclusion devient une égalité lorsque $H_{q+1}(E'_\bullet)=0$. On en déduit que, sous cette condition, si l'image de $d^q$ est fermée alors $H^{q+1}(E^\bullet)=0$.
\end{proof}

Nous pouvons maintenant étendre au cas de la cohomologie $L^p$, $p>1$, le Théorème 3.2 de \cite{LaShdualiteL2}

\begin{thm}\label{nonsep2}
Soient $X$ une variété analytique complexe de dimension complexe $n$, $D$ un domaine relativement compact de $X$ et $E$ un fibré holomorphe sur $X$. On suppose que, pour $0\leq r\leq n$, $H_c^{n-r,1}(X,E)=0$ et que $X\setminus D$ est connexe, alors

(i) Si $D$ est à bord rectifiable, soit $H^{r,n-1}_{L^p,s}(D,E)=0$, soit $H^{r,n-1}_{L^p,s}(D,E)$ n'est pas séparé.

(ii) Si $D$ est à bord Lipschitz, soit $H^{r,n-1}_{L^p}(D,E)=0$, soit $H^{r,n-1}_{L^p}(D,E)$ n'est pas séparé.
\end{thm}
\begin{proof}[Démonstration]
Le point (i) est une conséquence immédiate des Propositions \ref{degre1} et \ref{nonsep}.

Le point (ii) résulte du fait que si le bord de $D$ est Lipschitz, alors les groupes de cohomologie $H^{r,n-1}_{L^p,s}(D,E)$ et $H^{r,n-1}_{L^p}(D,E)$ coïncident.
\end{proof}

En associant le Théorème \ref{complexdual} et la Proposition \ref{nonsep} on obtient
\begin{cor}\label{annulation}
Soient $(E^\bullet,d)$ et $(E'_\bullet,d')$ deux complexes d'espaces de Banach réflexifs duaux dont les opérateurs ont un domaine dense. Supposons que $H_q(E'_\bullet)$ est séparé et que $H_{q+1}(E'_\bullet)=0$, alors $H^{q+1}(E^\bullet)=0$.
\end{cor}

Nous allons utiliser le Corollaire \ref{annulation} pour étudier le problème de Cauchy faible de résolution à support exact suivant~:
Soit $X$ une variété analytique complexe de dimension complexe $n\geq 2$, $E$ un fibré vectoriel holomorphe sur $X$ et $D$ un domaine relativement compact dans $X$. Etant donnée une $(r,q)$-forme $f$ à coefficients dans $L^p(X,E)$, avec $0\leq r\leq n$ et $1\leq q\leq n$, telle que
$${\rm supp}~f\subset\ol D \quad {\rm et}\quad \opa f=0~{\rm au~sens~des~distributions~dans~}X,$$
existe-t-il une $(r,q-1)$-forme $g$  à coefficients dans $L^p(X)$ telle que
$${\rm supp}~g\subset\ol D \quad {\rm et}\quad \opa g=f~{\rm au~sens~des~distributions~dans~}X~?$$

Observons que lorsque $q=n$, la condition $\opa f=0$ dans l'énoncé du problème de Cauchy est vide et qu'une condition nécessaire pour avoir une solution est que
$f$ satisfasse
$$\int_D f\wedge \varphi=0,~{\rm pour~toute~} \varphi\in L^{p'}_{n-r,0}(D,E^*)\cap Dom(\opa_s)\cap\ker \opa,$$
puisque $\im\opa_{\wt c} \subset(\ker~^t\opa_s)^\circ$ et qu'elle est suffisante si $\im \opa_{\wt c}$ est fermé puisque $\ol{\im \opa_{\wt c}}=(\ker~^t\opa_s)^\circ$. On déduit donc du Théorème \ref{complexdual}
\begin{thm}\label{n}
Soit $X$ une variété analytique complexe de dimension complexe $n\geq 2$, $E$ un fibré vectoriel holomorphe sur $X$ et $D$ un domaine relativement compact dans $X$. On suppose que $D$ est à bord Lipschitz et que $H^{n-r,1}_{L^{p'}}(D,E^*)$ est séparé alors le problème de Cauchy faible de résolution à support exact à une solution en degré $n$ si et seulement si la forme $f$ vérifie
$$\int_D f\wedge \varphi=0,~{\rm pour~toute~} \varphi\in L^{p'}_{n-r,0}(D,E^*)\cap\ker \opa.$$
\end{thm}

Une réponse affirmative au problème de Cauchy ci-dessus en degré $q$ pour $1\leq q\leq n-1$ équivaut à l'annulation du groupe de cohomologie $H^{r,q}_{\wt c,L^p}(D,E)$. Si de plus le bord de $D$ est Lipschitz, c'est équivalent à $H^{r,q}_{c,L^p}(D,E)=0$. Observons que dans ce cas, il résulte du théorème de l'application ouverte qu'il existe une constante $C>0$, indépendante de $f$, telle que la solution $g$ vérifie $\|g\|_p\leq C~\|f\|_p$.

En appliquant le Corollaire \ref{annulation} au complexe $(E^\bullet,d)$ avec, pour $r$ fixé tel que $0\leq r\leq n$, $E^q=L^{p'}_{r,q}(D,E)$ si $0\leq q\leq n$ et $E^q=\{0\}$ si $q<0$ ou $q>n$, et $d=\opa_{\wt c}$, on obtient~:
\begin{thm}\label{infn-1}
Soit $X$ une variété analytique complexe de dimension complexe $n\geq 2$, $E$ un fibré vectoriel holomorphe sur $X$, $D$ un domaine relativement compact dans $X$ et $q$ un entier compris entre $1$ et $n-1$. On suppose que $D$ est à bord Lipschitz et que $H^{n-r,n-q+1}_{L^{p'}}(D,E^*)$ est séparé et $H^{n-r,n-q}_{L^{p'}}(D,E^*)=0$, alors
$$H^{r,q}_{c,L^p}(D,E)=H^{r,q}_{\wt c,L^p}(D,E)=0.$$
\end{thm}

Il reste maintenant à donner des conditions géométriques sur $D$ permettant d'assurer que $H^{n-r,n-q+1}_{L^{p'}}(D,E^*)$ est séparé et que $H^{n-r,n-q}_{L^{p'}}(D,E^*)=0$ et donc que le problème de Cauchy faible de résolution à support exact à une solution en degré $q$.

Lorsque $X=\cb^n$, N. Ovrelid \cite{Ov} a prouvé que, pour $1\leq q\leq n$, $H^{r,q}_{L^{p'}}(D)=0$ lorsque $D$ est un domaine strictement pseudoconvexe borné à bord de classe $\ci$, ce résultat a été étendu au cas où $D$ est un polydisque par P. Charpentier \cite{Char} et aux intersections transverses finies de domaines strictement pseudoconvexes bornés par C. Menini \cite{Me}.
\begin{cor}\label{cn}
Si $D$ est un polydisque ou une intersection transverse finie de domaines strictement pseudoconvexes bornés de $\cb^n$ et $r$ un entier tel que $0\leq r\leq n$, alors pour tout $q$ tel que $1\leq q\leq n-1$ et toute $(r,q)$-forme $f$ à coefficients dans $L^p(X)$, telle que
$${\rm supp} f\subset\ol D \quad {\rm et}\quad \opa f=0~{\rm au~sens~des~distributions~dans~}X,$$
il existe une $(r,q-1)$-forme $g$  à coefficients dans $L^p(\cb^n)$ telle que
$${\rm supp} g\subset\ol D \quad {\rm et}\quad \opa g=f~{\rm au~sens~des~distributions~dans~}X.$$
De plus il existe une constante $C>0$, indépendante de $f$, telle que la solution $g$ vérifie $\|g\|_p\leq C~\|f\|_p$.
\end{cor}

Plus généralement L. Ma et S. Vassiliadou \cite{MaVa} ont prouvé que $H^{0,q}_{L^{p'}}(D)=0$ pour $q\geq s$, lorsque $D$ est une bonne intersection transverse finie de domaines strictement $s$-convexes de $\cb^n$, $1\leq s\leq n-1$, où $s=1$ correspond à strictement pseudoconvexe. Pour un tel domaine $D$, le problème de Cauchy faible de résolution à support exact aura donc une solution en degré $q$ si $1\leq q\leq n-s$.

Observons que dans le cas où $X$ est une variété analytique complexe et $q=1$, les conditions cohomologiques sur $D$ peuvent être remplacées par une condition cohomologique sur $X$ accompagnée d'une condition topologique sur $D$. A la fin de la section \ref{s1}, nous avons prouvé

\begin{prop}\label{degre1bis}
Soient $X$ une variété analytique complexe de dimension complexe $n$, $D$ un domaine relativement compact de $X$ et $E$ un fibré holomorphe sur $X$. On suppose que, pour $0\leq r\leq n$, $H_c^{n-r,1}(X,E)=0$ (ce qui est satisfait par exemple si $X$ est $(n-1)$-complète au sens d'Andreotti-Grauert et en particulier par les variétés de Stein de dimension $n\geq 2$) et que $X\setminus D$ est connexe, alors pour toute $(n-r,1)$-forme $f$ à coefficients dans $L^p(X)$ à support dans $\ol D$, il existe une $(n-r,0)$-forme $g$ à coefficients dans $L^p(X)$ et à support dans $\ol D$ telle que $\opa g=f$.
\end{prop}

Il résulte du Corollaire \ref{resolutionLp} que lorsque $X$ est une variété analytique complexe et $D\subset\subset X$ un domaine complètement strictement $s$-convexe de $X$ l'équation de Cauchy-Riemann peut-être résolue avec des estimations $L^{p'}$ sur $D$ en degré $q\geq s$.
C'est le cas en particulier lorsque $X$ est une variété de Stein et  $D$ un domaine strictement pseudoconvexe à bord lisse relativement compact dans $X$.

\begin{cor}\label{stein}
Si $X$ est une variété analytique complexe de dimension complexe $n\geq 2$, $D$ un domaine complètement strictement $s$-convexe à bord lisse relativement compact dans $X$ et $r$ un entier tel que $0\leq r\leq n$, alors pour tout $q$ tel que $1\leq q\leq n-s$ et toute $(r,q)$-forme $f$ à coefficients dans $L^p(X)$, telle que
$${\rm supp} f\subset\ol D \quad {\rm et}\quad \opa f=0~{\rm au~sens~des~distributions~dans~}X,$$
il existe une $(r,q-1)$-forme $g$  à coefficients dans $L^p(X)$ telle que
$${\rm supp} g\subset\ol D \quad {\rm et}\quad \opa g=f~{\rm au~sens~des~distributions~dans~}X.$$
\end{cor}

Il résulte du Théorème \ref{n} que les Corollaires \ref{cn} et \ref{stein} restent valables pour $q=n$ si $f$ satisfait la condition d'orthogonalité
$$\int_D f\wedge \varphi=0,~{\rm pour~toute~} \varphi\in L^{p'}_{n-r,0}(X)\cap\ker \opa.$$

Si le domaine $D$ est complètement strictement $s$-convexe, c'est un ouvert $s$-complet et lorsque le support de $f$ est contenu dans $D$, nous avons déjà prouvé à la fin de la section \ref{s1} qu'alors le support de $g$ est également contenu dans $D$. Notre nouveau résultat autorise le support de $f$ à rencontrer le bord de $D$, nous obtenons donc un contrôle beaucoup plus précis du support de la solution $g$ en fonction du support de la donnée $f$.
\medskip

Nous allons appliquer les résultats que nous venons d'obtenir à l'extension des formes différentielles $\opa$-fermées.

\begin{thm}\label{extension}
Soit $X$ une variété analytique complexe de dimension complexe $n\geq 2$, $E$ un fibré vectoriel holomorphe sur $X$, $D$ un domaine relativement compact dans $X$ et $q$ un entier compris entre $0$ et $n-1$. On suppose que $D$ est à bord Lipschitz et que, pour tout entier $r$, $0\leq r\leq n$, $H^{n-r,n-q}_{L^{p'}}(D,E^*)$ est séparé et $H^{n-r,n-q-1}_{L^{p'}}(D,E^*)=0$ si de plus $q\leq n-2$, alors

(i) Si $f\in W^{1,p}_{r,q}(X\setminus D,E)$, $0\leq q\leq n-2$, vérifie $\opa f=0$ sur $X\setminus D$, il existe $F\in L^p_{r,q}(X,E)$ telle que $F_{|_{X\setminus D}}=f$ et $\opa F=0$ sur $X$;

(ii) Le résultat de (i) reste vrai pour $q=n-1$ si $f$ satisfait
\begin{equation}\label{moment}
\int_{\pa D} f\wedge \varphi=0,~{\rm pour~toute~} \varphi\in L^{p'}_{n-r,0}(X,E^*)\cap\ker \opa.
\end{equation}
\end{thm}
\begin{proof}[Démonstration]
Puisque $D$ est à bord Lipschitz, il existe un opérateur d'extension continu $Ext~:~W^{1,p}_{r,q}(X\setminus D,E)\to W^{1,p}_{r,q}(X,E)$. Posons $\wt f=Ext(f)$, alors $\wt f_{|_{X\setminus D}}=f$ et $\opa\wt f\in L^p_{r,q+1}(X,E)$, de plus le support de $\opa\wt f$ est contenu dans $\ol D$. Si $1\leq q\leq n-2$, en appliquant le Théorème \ref{infn-1}, on obtient l'existence d'une $(r,q)$-forme $g$  à coefficients dans $L^p(X,E)$ telle que
${\rm supp} g\subset\ol D$  et $ \opa g=\opa\wt f$ au sens des distributions dans $X$.
La forme différentielle $F=\wt f-g$ satisfait la conclusion du théorème.

Si $q=n-1$, on a pour toute $\varphi\in L^{p'}_{n-r,0}(X,E^*)\cap\ker \opa$
\begin{equation*}
\int_D \opa\wt f\wedge\varphi=\lim_{\nu\to +\infty}\int_{D}\opa\wt f_\nu\wedge\varphi=\lim_{\nu\to +\infty}\int_D d(\wt f_\nu\wedge\varphi)=\lim_{\nu\to +\infty}\int_{\pa D} \wt f_\nu\wedge\varphi=\int_{\pa D} f\wedge\varphi=0,
\end{equation*}
par la formule de Stokes, où la suite $(\wt f_\nu)_{\nu\in\nb}$ est une suite de formes de classe $\ci$ sur $X$ qui converge vers $\wt f$ en norme $W^{1,p}$, obtenue grâce au lemme de Friedrich. On peut alors appliquer le Théorème \ref{n}.
\end{proof}

On en déduit facilement le corollaire suivant sur la résolution de l'opérateur de Cauchy-Riemann dans un anneau~:
\begin{cor}\label{anneau}
Soit $X$ une variété analytique complexe de dimension complexe $n\geq 2$, $E$ un fibré vectoriel holomorphe sur $X$, $D\subset\subset D_1$ deux domaines relativement compacts dans $X$ et $q$ un entier compris entre $1$ et $n-1$. On suppose que  $D$ est à bord Lipschitz et que, si $0\leq r\leq n$, $H^{r,q}_{L^{p}}(D_1,E^*)=0$, $H^{n-r,n-q}_{L^{p'}}(D,E^*)$ est séparé. 

Si $H^{n-r,n-q-1}_{L^{p'}}(D,E^*)=0$ et si $f\in W^{1,p}_{r,q}(D_1\setminus D,E)$, $1\leq q\leq n-2$, vérifie $\opa f=0$ sur $X\setminus D$, il existe une $(r,q-1)$-forme $g$  à coefficients dans $L^p(D_1\setminus D)$ telle que $\opa g=f$.

Si $q=n-1$, le résultat reste vrai sous réserve que $f$ satisfasse \eqref{moment}.

Si de plus le bord de $D_1$ est de classe $\ci$ et satisfait la condition $Z(q)$, il existe une $(r,q-1)$-forme $g$  à coefficients dans $W^{1,p}(D_1\setminus D,E)\cap W^{1/2,p}(\ol D_1\setminus D,E)$ telle que $\opa g=f$.

\end{cor}

Donnons des exemples de domaines $D$ et $D_1$ qui satisfont les hypothèses du Théorème \ref{extension} et du Corollaire \ref{anneau}.

Si $X=\cb^n$, $D$ et $D_1$ peuvent être chacun des domaines strictement pseudoconvexes bornés à bord $\ci$ ou des intersections transverses finies de tels domaines ou encore des polydisques. Pour un entier $q$ donné, $D$ et $D_1$ peuvent également être respectivement des domaines $s$-convexes et $s^*$-convexes ou des intersections finies transverses de tels domaines au sens de Ma et Vassiliadou avec $s\leq n-q-1$ et $s^*\leq q$.

Si $X$ est une variété analytique complexe de dimension complexe $n\geq 2$ et $q=0$, le Théorème  \ref{extension} est valide  si  $X$ est une variété de Stein et $X\setminus D$ est connexe (on retrouve ainsi le phénomène de Hartogs car $f$ est alors une fonction holomorphe) par la Proposition \ref{degre1bis}.

Si $X$ est une variété analytique complexe de dimension complexe $n\geq 2$ et $q\geq 1$, les domaines $D$ et $D_1$ peuvent être respectivement des domaines complètement strictement $s$-convexes et complètement strictement $s^*$-convexes à bord $\ci$ avec $s\leq n-q-1$ et $s^*\leq q$. Ces conditions seront en particulier vérifiées si $X$ est une variété de Stein et $D\subset\subset D_1\subset\subset X$ deux domaines strictement pseudoconvexes à bord $\ci$.

Les deux résultats précédents ont été prouvés par D. Chakrabarty et M.-C. Shaw dans le paragraphe 4 de \cite{ChaSh} lorsque $p=2$ en utilisant des poids singuliers. 
\medskip

Terminons cette section par une extension au cadre $L^p$ de la Proposition 4.7 et du Corollaire 4.8 démontrés dans \cite{LaShdualiteL2} pour $p=2$.
\begin{prop}
Soit $X$ une variété de Stein de dimension complexe $n\geq 2$, $E$ un fibré vectoriel holomorphe sur $X$ et $D$ un domaine relativement compact dans $X$ à bord Lipschitz tel que $X\setminus D$ soit connexe. Alors, pour tout $r$ tel que $0\leq r\leq n$, $H^{r,n}_{c,L^p}(X\setminus\ol D,E)$ est séparé si et seulement si $H^{r,n-1}_{W^{1,p}}(D,E)=0$, où $H^{r,n-1}_{W^{1,p}}(D,E)$ désigne le groupe de  cohomologie $(r,n-1)$ de $E$ à coefficients $W^{1,p}$.
\end{prop}
\begin{proof}[Démonstration]
Soit $f\in W^{1,p}_{r,n-1}(D,E)$ une forme $\opa$-fermée sur $D$, notons $\wt f$ une extension $W^{1,p}$ à support compact de $f$ à $X$. Une telle extension existe puisque le bord de $D$ est Lipschitz par le Théorème 1.4.3.1 de \cite{GrSobolev}. La $(r,n)$-forme $\opa\wt f$ est à coefficients $L^p$ et à support compact dans $X\setminus D$. De plus elle satisfait
$$\int_{X\setminus D} \theta\wedge\opa\wt f=0$$
pour toute $(n,0)$-forme holomorphe sur $X\setminus D$. En effet puisse que $X\setminus D$ est connexe et $X$ est une variété de Stein de  dimension complexe $n\geq 2$, le phénomène de Hartogs implique que $\theta$ s'étend en une $(n,0)$-forme $\wt\theta$ holomorphe sur $X$ et
$$\int_{X\setminus D} \theta\wedge\opa\wt f=\int_X \wt\theta\wedge\opa\wt f=\int_X \opa\wt\theta\wedge\wt f=0.$$
La propriété de séparation du groupe $H^{r,n}_{c,L^p}(X\setminus\ol D,E)=H^{r,n}_{\wt c,L^p}(X\setminus\ol D,E)$ implique alors qu'il existe une forme $g\in L^p_{r,n-1}(D,E)$ à support compact dans $X\setminus D$ telle que $\opa\wt f=\opa g$. par conséquent la forme $\wt f-g$ est une $(r,n-1)$-forme $\opa$-fermée sur $X$ dont la restriction à $D$ est égale à $f$. Puisque $X$ est une variété de Stein, il résulte de la section \ref{s1} que $H^{0,n-1}_{L^p_{loc}}(X,E)=0$. Il existe donc une forme $h\in (L^p_{loc})^{0,n-2}(X,E)$ telle que $\wt f-g=\opa h$ sur $X$. Par la régularité intérieure de l'opérateur $\opa$ (cf. Proposition \ref{reg}), on peut choisir $h$ telle que $h\in (W^{1,p})^{0,n-2}(D,E)$ et puisque $g$ est  à support compact dans $X\setminus D$, on a $f=\opa h$ sur $D$ et donc $H^{r,n-1}_{W^{1,p}}(D,E)=0$.

Réciproquement soit $f\in L^p_{r,n}(X,E)$ une forme $\opa$-fermée à support compact dans $X\setminus D$, orthogonale aux $(n-r,0)$-formes $L^{p'}$ à valeurs dans $E^*$ et holomorphes dans $X\setminus\ol D$ et en particulier aux $(n-r,0)$-formes $L^{p'}$ à valeurs dans $E^*$ et holomorphes dans $X$. Puisque $X$ est une variété de Stein, il résulte de la dualité de Serre que $H^{r,n}_{c,L^p}(X,E)$ est séparé et par conséquent il existe une $(0,n-1)$-forme $g\in  L^p_{r,n-1}(X,E)$ à support compact dans $X$ telle que $f=\opa g$. De nouveau la  régularité intérieure de l'opérateur $\opa$ implique que l'on peut choisir $g$ dans $W^{1,p}_{r,n-1}(X,E)$. Comme le support de $f$ est contenu dans $X\setminus D$, la forme $g$ est $\opa$-fermée dans $D$ et puisque $H^{r,n-1}_{W^{1,p}}(D,E)=0$, on a $g=\opa h$ pour une $(r,n-2)$-forme $h$ dans $W^{1,p}_{r,n-2}(D,E)$. Soit $\wt h$ une extension $W^{1,p}$ de $h$ à support compact dans $X$, alors $g-\opa\wt h$ est dans $L^p_{r,n-1}(X,E)$, à support compact dans $X$, s'annule sur $D$ et vérifie $\opa(g-\opa\wt h)=f$, ce qui prouve que $H^{r,n}_{c,L^p}(X\setminus\ol D,E)$ est séparé.
\end{proof}

En utilisant la dualité entre les complexes $\big((L^p_{r,\bullet}(D,E),\opa)$ et $(L^{p'}_{n-r,\bullet}(D,E^*),\opa_c)\big)$, on obtient
\begin{cor}
Soit $X$ une variété de Stein de dimension complexe $n\geq 2$, $E$ un fibré vectoriel holomorphe sur $X$ et $D$ un domaine relativement compact dans $X$ à bord Lipschitz tel que $X\setminus D$ soit connexe. Alors, pour tout $r$ tel que $0\leq r\leq n$, soit $H^{r,n-1}_{W^{1,p}}(D,E)=0$, soit $H^{n-r,1}_{L^{p'}}(X\setminus\ol D,E^*)$ n'est pas séparé.
\end{cor} 

\providecommand{\bysame}{\leavevmode\hbox to3em{\hrulefill}\thinspace}
\providecommand{\MR}{\relax\ifhmode\unskip\space\fi MR }
\providecommand{\MRhref}[2]{%
  \href{http://www.ams.org/mathscinet-getitem?mr=#1}{#2}
}
\providecommand{\href}[2]{#2}

\enddocument

\end